\documentclass[final]{siamltex}
\usepackage{amsmath}
\usepackage{amssymb}
\usepackage{graphicx}
\usepackage{float}

\newcommand{\Z}{{\mathbb Z}}

\newtheorem{claim}{Claim}

\newcommand{\vsten}{\vspace{10pt}}

\newenvironment{proofcite}[1]{\noindent{\bf Proof of #1.\,}}{\hfill$\Box$}

\begin{document}
\title{Improved Bounds for $r$-Identifying Codes of the Hex Grid}
\author{Brendon Stanton\thanks{Iowa State University, Department of Mathematics, 396 Carver Hall, Ames IA, 50010}}
\maketitle

\begin{abstract}For any positive integer $r$, an $r$-identifying code on a graph $G$ is a set $C\subset V(G)$ such that for every vertex in $V(G)$, the intersection of the radius-$r$ closed neighborhood with $C$ is nonempty and pairwise distinct.  For a finite graph, the density of a code is $|C|/|V(G)|$, which naturally extends to a definition of density in certain infinite graphs which are locally finite.  We find a code of density less than $5/(6r)$, which is sparser than the prior best construction which has density approximately $8/(9r)$.
\end{abstract}

\section{Introduction}

Given a connected, undirected graph $G=(V,E)$, define
$B_r(v)$, called the ball of radius $r$ centered at $v$, to be
$$B_r(v)=\{u\in V(G): d(u,v)\le r\} $$ where $d(u,v)$ is the distance between
$u$ and $v$ in $G$.

Let $C\subset V(G)$.  We say that $C$ is an $r$-identifying code if $C$ has the properties:
\begin{enumerate}
\item $B_r(v) \cap C \neq \emptyset, \text{ for all } v\in V(G) \text{ and}$
\item $B_r(u) \cap C \neq B_r(v)\cap C, \text{ for all distinct } u,v\in V(G).$
\end{enumerate}
It is simply called a code when
$r$ is understood.  The elements of $C$ are called codewords.
We define $I_r(v)=I_r(v,C)=B_r(v)\cap C$.  We call
$I_r(v)$ the identifying set of $v$ with respect to $C$.  If $I_r(u)\neq I_r(v)$ for some $u\neq v$, the we say $u$ and $v$ are distinguishable.  Otherwise, we say they are indistinguishable.

Vertex identifying codes were introduced in ~\cite{Karpovsky1998} as
a way to help with fault diagnosis in multiprocessor computer
systems.  Codes have been studied in many graphs.  Of particular
interest are codes in the infinite triangular, square, and hexagonal
lattices as well as the square lattice with diagonals (king grid).
We can define each of these graphs so that they have vertex set
$\Z\times\Z$.  Let $Q_m$ denote the set of vertices $(x,y)\subset
\Z\times\Z$ with $|x|\le m$ and $|y|\le m$.  The
density of a code $C$ defined in ~\cite{Charon2002} is
$$D(C)=\limsup_{m\rightarrow\infty}\frac{|C\cap Q_m|}{|Q_m|}.$$

When examining a particular graph, we are interested in finding the
minimum density of an $r$-identifying code.  The exact minimum
density of an $r$-identifying code for the king grid has been found in
~\cite{Charon2004}.  General constructions of $r$-identifying codes
for the square and
triangular lattices have been given in ~\cite{Honkala2002} and
~\cite{Charon2001}.

For this paper, we focus on the hexagonal grid.  It was shown in
~\cite{Charon2001} that $$\frac{2}{5r}-o(1/r)\le D(G_H,r) \le
\frac{8}{9r}+o(1/r).$$  Where $D(G_H,r)$ represents the minimum density of
an $r$-identifying code in the hexagonal grid.
The main theorem of this paper is Theorem ~\ref{maintheorem}:

\begin{theorem}\label{maintheorem} There exists an $r$-identifying code of
density $$\frac{5r+3}{6r(r+1)},  \text{ if $r$ is even} ;\qquad
            \frac{5r^2+10r-3}{(6r-2)(r+1)^2},  \text{ if $r$ is odd}.$$
\end{theorem}

The proof of Theorem ~\ref{maintheorem} can be found in Section ~\ref{mainthmsection}.  Section ~\ref{descriptionsection} provides an brief description of a code with the aforementioned density and gives a few basic definitions needed to describe the code.  Section ~\ref{distanceclaimssection} provides a few technical claims needed for the proof of Theorem ~\ref{maintheorem} and the proofs of these claims can be found in Section ~\ref{lemmaproofs}.

\section{Construction and Definitions}\label{descriptionsection}

\begin{figure}[ht]
\centering
\includegraphics[totalheight=0.3\textheight]{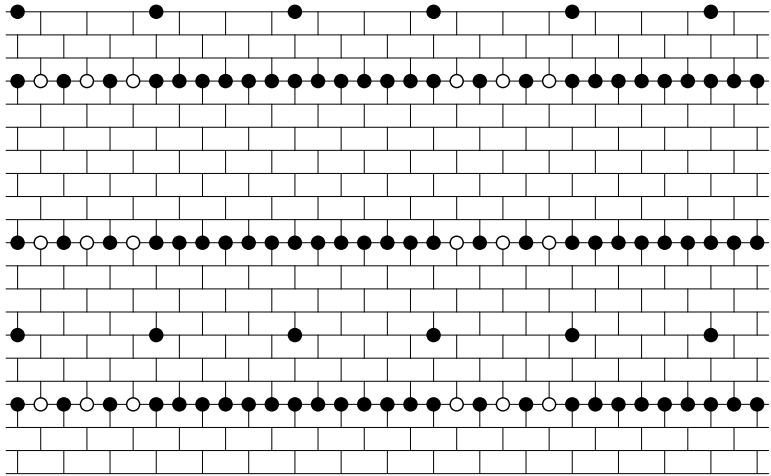}
\caption[The code $C=C'\cup C''$ for $r=6$.  Black vertices are code
words. White vertices are vertices in $L_{n(r+1)}$ which are not in
$C$.] {The code $C=C'\cup C''$ for $r=6$.  Black vertices are code
words. White vertices are vertices in $L_{n(r+1)}$ which are not in
$C$.}\label{fig:code6}
\end{figure}

\begin{figure}[ht]
\centering
\includegraphics[totalheight=0.3\textheight]{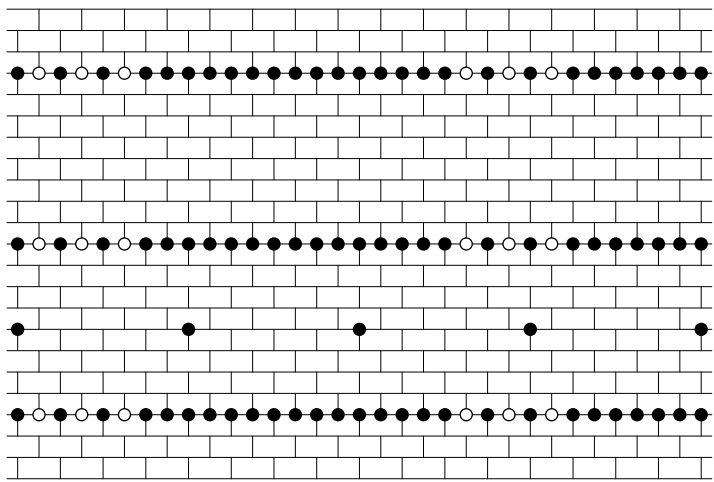}
\caption[The code $C=C'\cup C''$ for $r=7$.  Black vertices are code
words. White vertices are vertices in $L_{n(r+1)}$ which are not in
$C$.] {The code $C=C'\cup C''$ for $r=7$.  Black vertices are code
words. White vertices are vertices in $L_{n(r+1)}$ which are not in
$C$.}\label{fig:code7}
\end{figure}

For this construction we will use the brick wall representation of
the hex grid.  To describe this representation, we need to briefly consider the square grid $G_S$.
The square grid has vertex set $V(G_S)=\Z\times\Z$ and
$$E(G_S)=\{\{u=(i,j),v\}: u-v\in \{(0,
\pm 1),(\pm 1,0)\}\}.$$

Let $G_H$ represent the hex grid.  Then $V(G_H)=\Z\times\Z$ and
$$E(G_H)=\{\{u=(i,j),v\}: u-v\in \{(0,(-1)^{i+j+1}),(\pm 1,0)\}\}.$$

In other words, if $x+y$ is even, then $(x,y)$ is adjacent to
$(x,y+1),(x-1,y),$ and $(x+1,y)$.  If $x+y$ is odd, then $(x,y)$ is
adjacent to $(x,y-1),(x-1,y),$ and $(x+1,y)$. However, the first
representation shows clearly that the hex grid is a subgraph of the
square grid.

For any integer $k$, we also define a \textit{horizontal line} $L_k=\{(x,k): x\in \Z \}$.

Note that if $u,v\in V(G_S)$, then the distance between them (in
the square grid) is $\|u-v\|_1$.  From this point forward, let $d(u,v)$
represent the distance between two vertices in the hex grid.  If
$u\in V(G_H)$ and $U,V\subset(G_H)$, we define $d(u,V)=\min
\{d(u,v):v\in V\}$ and $d(U,V)=\min \{d(u,V): u\in U\}$.

Define $$L_k'=\left\{\begin{array}{cc}
                       L_k\cap \{(x,k):x\not\equiv 1,3,5,\ldots,r-1\mod 3r\},& \text{if $r$ is even}; \\
                       L_k\cap \{(x,k):x\not\equiv 1,3,5,\ldots,r-1\mod 3r-1\},& \text{if $r$ is odd} \\
                     \end{array}\right.$$
and with $\delta=0$ if $k$ is even and $\delta=1$ otherwise define $$L_k''=\left\{\begin{array}{cc}
                       L_k\cap \{(x,k):x\equiv \delta \mod r\},& \text{if $r$ is even}; \\
                       L_k\cap \{(x,k):x\equiv 0\mod r+1\},& \text{if $r$ is odd} \\
                     \end{array}\right..$$

Finally, let
$$C'=\bigcup_{n=-\infty}^\infty L_{n(r+1)}' \qquad\text{and}\qquad
C''=\bigcup_{n=-\infty}^\infty L_{\lfloor (r+1)/2\rfloor+2n(r+1)}''.$$
Let $C=C'\cup C''$.  We will show in Section ~\ref{mainthmsection} that $C$ is a valid code of the density described in Theorem ~\ref{maintheorem}.  Partial pictures of the code are shown for $r=6$ in Figure ~\ref{fig:code6} and $r=7$ in Figure ~\ref{fig:code7}.

\section{Distance Claims}\label{distanceclaimssection}

We present a list of lemmas on the distances of vertices in the hex
grid.  These lemmas will be used in the proof of our construction.
The proofs of these lemmas can be found in Section
\ref{lemmaproofs}.\vsten

\begin{claim}\label{taxicab1} For $u,v\in V(G_H)$, $d(u,v)\ge
\|u-v\|_1$.\end{claim}

\begin{proof}
Note that Lemma ~\ref{taxicab1} simply says that the distance
between any two vertices in our graph is at least as long as the
their distance in the square grid.  Since $G_H$ is a subgraph of
$G_S$, any path between $u$ and $v$ in $G_H$ is also a path in $G_S$
and so the claim follows.
\end{proof}\vsten

\begin{claim}\label{taxicab2} \textbf{(Taxicab Lemma)}
For $u=(x,y),v=(x'y')\in V(G_H)$, if $|x-x'|\ge|y-y'|$, then
$d(u,v)=\|u-v\|_1$.\end{claim}

This fact is used so frequently that we give it the name the Taxicab
Lemma.  It states that if the horizontal distance between two
vertices is at least as far as the vertical distance, then the
distance between those two vertices is exactly the same as it would
be in the square grid.

The proof of the Taxicab Lemma and the remainder of these claims can be found in
Section ~\ref{lemmaproofs}.

Claims ~\ref{vertexlinedistance} and ~\ref{vertexlinedistance2} say that the distance
between a point $(k,a)$ and a line $L_b$ is either $2|a-b|$ or $2|a-b|-1$ depending on various factors.  It also follows from these claims that $d(L_a,L_b)=2|a-b|-1$ if $a\neq b$. \vsten

\begin{claim}\label{vertexlinedistance}
Let $a<b$, then $$d((k,a),L_b)=\left\{\begin{array}{cc}
                                              2(b-a)-1, & \text{if $a+k$ is even}; \\
                                              2(b-a), & \text{if $a+k$ is odd}.
                                            \end{array}
\right.$$
\end{claim} \vsten

\begin{claim}\label{vertexlinedistance2} Let $a>b$, then $$d((k,a),L_b)=\left\{\begin{array}{cc}
                                              2(a-b), & \text{if $a+k$ is even}; \\
                                              2(a-b)-1, & \text{if $a+k$ is odd}.
                                            \end{array}
\right.$$
\end{claim}

The next three claims all basically have the same idea.  If we are looking at a point $(x,y)$ and a horizontal line $L_k$ such that $(x,y)$ is within some given distance $d$, we can find a sequence $S$ of points on this line such that each point is at most distance $d$ from $(x,y)$ and the distance between each point in $S$ and its closest neighbor in $S$ is exactly 2. \vsten

\begin{claim}\label{evenvertexdistance} Let $k$ be a positive integer.
There exist paths of length $2k$ from $(x,y)$ to $v$ for each $v$ in
$$\{(x-k+2j,y\pm k): j=0,1,\ldots,k\}$$
\end{claim} \vsten

\begin{claim}\label{oddevenvertexdistance} Let $k$ be a positive integer and let $x+y$ be even.
There exist paths of length $2k+1$ from $(x,y)$ to $v$ for each $v$
in
$$\{(x-k+2j,y+k+1): j=0,1,\ldots,k\}\cup\{(x-k-1+2j,y-k): j=0,1,\ldots,k+1\}$$
\end{claim} \vsten

\begin{claim}\label{oddoddvertexdistance} Let $k$ be a positive integer and let $x+y$ be odd.
There exist paths of length $2k+1$ from $(x,y)$ to $v$ for each $v$
in
$$\{(x-k+2j,y-k-1): j=0,1,\ldots,k\}\cup\{(x-k-1+2j,y+k): j=0,1,\ldots,k+1\}$$
\end{claim}

In the final claim, we are simply stating that if we are given a point $(x,y)$ and a line $L_k$ such that $d((x,y),L_k)<r$, we can find a path of vertices on that line which are all distance at most $r$ from $(x,y)$. \vsten

\begin{claim}\label{vertexballdistance} Let $(x,y)$ be a vertex and
$L_k$ be a line.  If $d((x,y),L_k)< r$, then
$$\{(x-(r-|y-k|)+j,k): j=0,1,\ldots, 2(r-|y-k|)\}\subset B_r((x,y))$$
\end{claim}

\section{Proof of Main Theorem}\label{mainthmsection}Here, we wish to show that the set described in Section~\ref{descriptionsection} is indeed a valid $r$-identifying code. \vsten

\begin{proofcite}{Theorem~\ref{maintheorem}}

Let $C=C'\cup C''$ be the code described in Section ~\ref{descriptionsection}.
 Let $d(C)$ be the density of $C$ in $G_H$,
$d(C')$ be the density of $C'$ in $G_H$ and $d(C'')$ be the density
of $C''$ in $G_H$.  Since $C'$ and $C''$ are disjoint, we see that

\begin{eqnarray*}
  d(C) &=& \limsup_{m\rightarrow\infty}\frac{|C\cap G_m|}{|G_m|} \\
  &=& \limsup_{m\rightarrow\infty}\frac{|(C'\cup C'')\cap
  G_m|}{|G_m|} \\
  &=& \limsup_{m\rightarrow\infty}\frac{|C'\cap
  G_m|}{|G_m|}+\limsup_{m\rightarrow\infty}\frac{|C''\cap
  G_m|}{|G_m|} \\
  &=& d(C')+d(C'')
\end{eqnarray*}

It is easy to see that both $C'$ and $C''$ are periodic tilings of
the plane (and hence $C$ is also periodic).  The density of periodic
tilings was discussed By Charon, Hudry, and
Lobstein~\cite{Charon2002} and it is shown that the density of a
periodic tiling on the hex grid is
$$D(C)=\frac{\text{\# of codewords in tile}}{\text{size of tile}} .$$

For $r$ even, we may consider the density of $C'$ on the tile
$[0,3r-1]\times[0,r]$. The size of this tile is $3r(r+1)$. On this
tile, the only members of $C'$ fall on the horizontal line $L_0$,
of which there are $2r+r/2$.

For $r$ odd, we may consider the density of $C'$ on the tile
$[0,3r-2]\times[0,r]$. The size of this tile is $(3r-1)(r+1)$.  On
this tile, the only members of $C'$ fall on $L_0$, of which there
are $2r+(r-1)/2$. Thus

$$d(C')=\left\{\begin{array}{cc}
                \frac{5}{6(r+1)} & \text{if $r$ is even} \\
                \frac{5r-1}{(6r-2)(r+1)} & \text{if $r$ is odd}
              \end{array}
\right.$$

For $C''$, we need to consider the tiling on $[0,r-1]\times[0,2r+1]$
if $r$ is even and $[0,r]\times[0,2r+1]$ if $r$ is odd.  In either
case, the tile contains only a single member of $C''$.  Hence,
$$d(C'')=\left\{\begin{array}{cc}
                \frac{1}{2r(r+1)} & \text{if $r$ is even} \\
                \frac{1}{2(r+1)^2} & \text{if $r$ is odd}
              \end{array}
\right.$$  Adding these two densities together gives us the numbers
described in the theorem.\vsten

It remains to show that $C$ is a valid code.  We will say that a
vertex $v$ is \emph{nearby} $L_k'$ if $I_r(v)\cap L_k'\neq
\emptyset$. An outline of the proof is as follows:\vsten
\begin{enumerate}
  \item Each vertex $v\in V(G_H)$ is nearby $L_{n(r+1)}'$ for
  exactly one value of $n$ (and hence, $I_r(v)\neq\emptyset$).
  \item Since each vertex is nearby $L_{n(r+1)}'$ for exactly one value
  of $n$, we only need to distinguish between vertices which are
  nearby the same horizontal line.
  \item The vertices in $C''$ distinguish between vertices that fall
  above the horizontal line and those that fall below the line.
  \item We then show that $L_{n(r+1)}'$ can distinguish between any
  two vertices which fall on the same side of the horizontal line (or in the
  line).
\end{enumerate}\vsten

\textbf{Part (1)}: We want to show that each vertex is nearby $L_{n(r+1)}$ for
exactly one value of $n$.

From Claim \ref{vertexlinedistance}, it immediately follows that
$d(L_{n(r+1)},L_{(n+1)(r+1)})=2r+1$.  So by the triangle inequality,
no vertex can be within a distance $r$ of both of these lines. Thus,
no vertex can be nearby more than one horizontal line of the form
$L_{n(r+1)}$.

\begin{claim}\label{codewordclaim}
Let $u,v\in L_{n(r+1)}$. Then:

\begin{enumerate}
  \item If $u\sim v$ then at least one of $u$ and $v$ is in $C'$.
  \item If $r\le d(u,v)\le 2r+1$ then at least one of $u$ and $v$
  is in $C'$.
\end{enumerate}
  Furthermore, note that for any $x$ at least one of the following
vertices is in $C$:
$$\{(x+2k,n(r+1)): k=0,1,\ldots,\lceil(r+1)/2\rceil\} .$$
\end{claim}

The proof of this claim is in Section ~\ref{lemmaproofs}.

From this point forward, let $v=(i,j)$ with $n(r+1)\le j <
(n+1)(r+1)$.

First, suppose that $r$ is even.  Then for $j\le r/2+ n(r+1)$ we have
$d((i,j) ,(i-r/2,j-r/2))=r$ by the Taxicab Lemma. Since $j-r/2\le
n(r+1)$, there is $u\in L_{n(r+1)}$ such that $d((i,j),u)\le r$.  If
$d((i,j),u)<r$, then $u$ and $u+(1,0)$ are both within distance $r$
of $(i,j)$ and by Claim ~\ref{codewordclaim} one of them is in $C'$.  If
$d((i,j),L_{n(r+1)})=r$, then by Claim \ref{evenvertexdistance}
$(i,j)$ is within distance $r$ of one of the vertices in the set
$\{(i-r/2+2k: k=0,1,\ldots,r/2\}$ and again by Claim ~\ref{codewordclaim} one of them must be in $C'$.
By a symmetric argument, we can show that if $j>r/2$, then there is
a $u\in C'\cap L_{(n+1)(r+1)}$ such that $d((i,j),u)\le r$.

If $r$ is odd, and $j\neq (r+1)/2$, then we can use the same
argument to show there is a codeword in $C'\cap L_{(n+1)(r+1)}$ or
$C'\cap L_{(n+1)(r+1)}$ within distance $r$ of $(i,j)$.  If
$j=(r+1)/2$, then we refer to Claims \ref{oddevenvertexdistance},
\ref{oddoddvertexdistance} and \ref{codewordclaim} to find an appropriate codeword.  Since
no vertex can be nearby more than one horizontal line of the form
$L_{n(r+1)}$, this shows that each vertex is nearby exactly one
horizontal line of this form.

So we have shown that $v$ is nearby exactly one horizontal line of
the form $L_{n(r+1)}'$, but this also shows that  $I_r(v)\neq
\emptyset$ for any $v\in V(G_H)$.

\textbf{Part (2)}: We now turn our attention to showing
that $I_r(u)\neq I_r(v)$ for any $u\neq v$. We first note that if $u$ is nearby $L_{n(r+1)}'$ and $v$ is nearby
$L_{m(r+1)}'$ and $m\neq n$, then there is some c in $I_r(v)\cap
L_{m(r+1)}'$ but $c\not\in I_r(u)$ since $u$ is not nearby
$L_{m(r+1)}'$.  Thus, $u$ and $v$ are distinguishable.
Hence it suffices to consider only the case that
$n=m$.

\textbf{Part (3)}:  We consider the case that $u$ and $v$ fall on opposite sides
of $L_{n(r+1)}$.  Suppose that $u=(i,j)$ where $j>n(r+1)$ and
$v=(i',j')$ where $j'<n(r+1)$. We see that if $n=2k$, then
$n(r+1)<\lfloor(r+1)/2\rfloor+2k(r+1)< (n+1)(r+1)$ and if $n=2k+1$
then $(n-1)(r+1)<\lfloor(r+1)/2\rfloor+2k(r+1)< n(r+1)$.  In the
first case, we see that
\begin{eqnarray*}
    d((i',j'),L_{\lfloor(r+1)/2\rfloor+2k(r+1)}) &\ge& d(L_{2k(r+1)-1},L_{\lfloor(r+1)/2\rfloor+2k(r+1)})\\
    &=&  2\lfloor (r+1)/2\rfloor + 1\\
    &\ge& r+1
\end{eqnarray*}
and likewise in the second case we have that $d((i,j),
L_{\lfloor(r+1)/2\rfloor+2k(r+1)})\ge r+1$ and so it follows that at
most one of $(i,j)$ and $(i',j')$ is within distance $r$ of a code
word in $C''$.

If $r$ is odd, then $\lfloor(r+1)/2\rfloor = (r+1)/2$ and
$|j-(r+1)/2+2k(r+1)|\le(r-1)/2$.  Hence, we see that
$d((i,j),L_{(r+1)/2+2k(r+1)})\le r-1$ .  From Claim
\ref{vertexballdistance}, it follows that $$\{(m,(r+1)/2+n(r+1)):
(r-1)/2\le m \le (3r-1)/2\}\subset B_r((i,j)) .$$  We note that the
cardinality of this set is at least $r+1$ and so in the case that
$n$ is even we see that at least one of these must be in $C''$.  A
symmetric argument applies if $n$ is odd, showing that $(i',j')$ is
within distance $r$ of a codeword in $C''$.

If $r$ is even, we apply a similar argument to show that $(i,j)$ is
within distance $r-1$ of $r$ vertices in
$L_{n(r+1)+\lfloor(r+1)/2\rfloor}$ and $(i',j')$ is within distance
$r$ of $r-1$ vertices of $L_{n(r+1)-\lceil(r+1)/2\rceil}$.  If $n$
is even, then clearly $(i,j)$ is within distance $r$ of some code
word in $C''$. However, if $n$ is odd, we have only shown that
$$\{(m,\lceil(r+1)\rceil/2+n(r+1)): r/2-2\le m \le 3r/2-1\}\subset
B_{r-1}((i,j)) .$$  Consider the set
$$\{(m,\lfloor(r+1)\rfloor/2+n(r+1)): r/2-2\le m \le 3r/2-1\}.$$  This set has $r$ vertices and so 1 of them must
be a codeword.  Furthermore, by the way we constructed $C''$, the
sum of the coordinates of codeword are even and so there is an edge
connecting it to a vertex in $$\{(m,\lceil(r+1)\rceil/2+n(r+1)):
r/2-2\le m \le 3r/2-1\}$$ and so it is within radius $r$ of
$(i',j')$.  In either case, we have exactly one of $u$ and $v$ is
within distance $r$ of a codeword in $C''$.

\textbf{Part (4)}: Now we have shown that two vertices are distinguishable if they are
nearby two different lines in $C'$ or if they fall on opposite sides
of a line $L_{n(r+1)}$ for some $n$.  To finish our proof, we need
to show that we can distinguish between $u=(i,j)$ and $v=(i',j')$ if
$u$ and $v$ are nearby the same line $L_{n(r+1)}'$ and $u$ and $v$
fall on the same side of that line.  Without loss of generality,
assume that $j,j'\ge n(r+1)$.

\textbf{Case 1}: $u,v\in L_{n(r+1)}$

Without loss of generality, we can write $u=(x,n(r+1))$ and
$v=(x+k,n(r+1))$ for $k>0$.  If $k>1$ then $(x-r,n(r+1))$ and
$(x-r+1,n(r+1))$ are both within distance $r$ of $u$ but not of $v$
and one of them must be a codeword by Claim \ref{codewordclaim}.  If $k=1$, then $(x-r,n(r+1))$
is within distance $r$ of $u$ but not $v$ and $(x+r+1,n(r+1))$ is
within distance $r$ of $v$ but not $u$. Since
$d((x-r,n(r+1)),(x+r+1,n(r+1)))=2r+1$, Claim \ref{codewordclaim} states that at least one of them must be a
codeword.

\textbf{Case 2}: $u\in L_{n(r+1)}$, $v\not\in L_{n(r+1)}$

Without loss of generality, assume that $i\le i'$.  We have
$u=(i,n(r+1))$.  If $i=i'$ then $v=(i,n(r+1)+k)$ for some $k>0$.
Consider the vertices $(i-r,n(r+1))$ and $(i+r,n(r+1))$.  These are
each distance $r$ from $u$ and they are distance $2r$ from each
other, so by Claim \ref{codewordclaim} at least one of them is a codeword.  However, by Claim
\ref{taxicab1}, these are distance at least $r+k>r$ from $v$ and so
$I_r(u)\neq I_r(v)$.

If $i<i'$, then we can write $v=(i+j, n(r+1)+k)$ for $j,k>0$.
Consider the vertices $x_1=(i-r, n(r+1))$ and $x_2=(i-r+1,n(r+1))$.
By the Taxicab Lemma, $d(u,x_1)=r$ and $d(u,x_2)=r-1$.  Further,
since they are adjacent vertices in $L_{n(r+1)}$, one of them is in
$C'$ by Claim \ref{codewordclaim}.  However, by Claim \ref{taxicab1} we have $d(v,x_1)\ge
r+j+k>r$ and $d(v,x_2)\ge r-1+j+k>r$.  Hence, neither of them is in
$I_r(v)$ and so $I_r(u)\neq I_r(v)$.

\textbf{Case 3}: $u,v\not\in L_{n(r+1)}$

Assume without loss of generality that $n=0$ and so $1\le j,j'\le
\lceil r/2\rceil$.

\textbf{Subcase 3.1}: $j<j'$, $i=i'$

From Claim \ref{vertexlinedistance2} it immediately follows that
$d((i,j),L_0)<d((i',j'),L_0)\le r$.  By the Taxicab Lemma, we have
$d((i-(r-j),0),(i,j))=d((i+(r-j),0),(i,j))=r$. Further note that
$d((i-(r-j),(i+(r-j),0))=2(r-j)$ and since $1\le j<\lceil r/2\rceil$
we have $r+1\le 2(r-j)\le 2r-2$ and so at least one of these two vertices is
a codeword by Claim \ref{codewordclaim}.  By the Taxicab Lemma, neither of these two vertices is
in $B_r((i',j'))$.

\textbf{Subcase 3.2}: $j<j'$, $i\neq i'$

Without loss of generality, we may assume that $i<i'$. Then, we wish
to consider the vertices $(i-(r-j),0)$ and $(i-(r-j)+1,0)$.  We see
that $d((i,j),(i-(r-j)+1,0))=r-1$ and so by the Taxicab Lemma,
neither of these vertices is in $B_r((i',j'))$.  But since these
vertices are adjacent and both in $L_0$ at least one of them is a
codeword by Claim \ref{codewordclaim} and so that one is in $I_r((i,j))$ but not in
$I_r((i',j'))$.

\textbf{Subcase 3.3}: $j=j'$, $d(u,L_0)<r$ or $d(v,L_0)<r$

Without loss of generality, assume that $d((i,j),L_0)<r$ and that
$i<i'$.  Then $v=(i+k,j)$ for some $k>0$.

First suppose that $k>1$. Then, we wish to consider the vertices
$(i-(r-j),0)$ and $(i-(r-j)+1,0)$. We see that
$d((i,j),(i-(r-j)+1,0))=r-1$ and so by the Taxicab Lemma, neither of
these vertices is in $B_r((i',j'))$. But since these vertices are
adjacent and both in $L_0$ so by Claim \ref{codewordclaim} at least one of them is a codeword and so
that one is in $I_r((i,j))$ but not in $I_r(i',j'))$.

If $k=1$, consider the vertices $x_1=(i-(r-j),0)$ and
$x_2=(i+1+(r-j),0)$.  We see that
$$\begin{array}{ll}
  d(x_1,u) &= r \\
  d(x_1,v) &= r+1 \\
  d(x_2,u) &= r+1 \\
  d(x_2,v) &= r
\end{array}$$
all by the taxicab lemma.  Furthermore, $d(x_1,x_2)=2(r-j)+1$ and so
$r+1\le d(x_1,x_2)\le 2r+1$ as in the above case and so one of them is
a code word by Claim \ref{codewordclaim}.  Hence $I_r(u)\neq I_r(v)$.

\textbf{Subcase 3.4}: $j=j'$, $d(u,L_0)=d(v,L_0)=r$

Without loss of generality, suppose $i'<i$.  We wish to consider the vertices
in the two sets
$$U=\{(i-(r-j)+2k,0): k=0,1,2,\ldots,r-j\}$$ and
$$V=\{(i'-(r-j)+2k,0):k=0,1,2,\ldots,r-j\}.$$  Note that
$$|U|=|V|=\left\{\begin{array}{cc}
                         r/2+1 & \text{if $r$ is even} \\
                         (r+1)/2 & \text{if $r$ is odd}
                       \end{array}
\right.$$

From Claims \ref{evenvertexdistance},
\ref{oddevenvertexdistance}, and \ref{oddoddvertexdistance} we have
$U\subset B_r(u)$ and $V\subset B_r(v)$. Each of $U$ and $V$ contain a codeword by Claim \ref{codewordclaim}.  Hence, if $U\cap
V=\emptyset$  then  $u$ and
$v$ are distinguishable.



If $U\cap V\neq \emptyset$ then the leftmost vertex in $U$ is also in $V$.  We further note that the leftmost vertex in $U$ is not the leftmost vertex in $V$ or else $U=V$ and $u=v$.  Thus, we must have $i'-(r-j)+2\le i\le i'+(r-j)$ and so $2\le i-i'\le 2(r-j)$.  Let
$x_1=(i+(r-j),0)$ and let $x_2=(i'-(r-j),0)$.  By definition $x_1\in
U$ and $x_2\in V$.  By the Taxicab Lemma we have $d(x_1,v)=|i+r-j-i'|+|j|=r+i-i'>r$ and likewise $d(x_2,u)>r$ so $x_1\not\in I_r(v)$ and $x_2\not\in I_r(u)$.  Also we have $d(x_1,x_2)=2(r-j)+i-i'$ which gives $2(r-j)+2\le d(x_1,x_2)\le 4(r-j)$.  Since $d(u,L_0)=r$ we must have $r-j=r/2$ if $r$ is even and $r-j=(r-1)/2$ if $r$ is odd by Claims ~\ref{vertexlinedistance} and ~\ref{vertexlinedistance2}.  In either case this gives
$r+2\le d(x_1,x_2)\le 2r$. By Claim \ref{codewordclaim}, one of $x_1$ and $x_2$ is a
codeword and so $u$ and $v$ are distinguishable.

This completes the proof of Theorem \ref{maintheorem}.
\end{proofcite}

\section{Proof of Claims}\label{lemmaproofs}

We now complete the proof by going back and finishing off the proofs of the claims and lemmas that were presented in Section \ref{distanceclaimssection}.\vsten

\begin{proofcite}{Claim~\ref{taxicab2}}\textbf{(Taxicab Lemma)}
By Claim ~\ref{taxicab1}, it suffices to show there exists a path of length $\|u-v\|$ between $u$ and $v$.  Since $|x-x'|\ge|y-y'|$, either $u=v$ or $x\neq x'$. If
$u=v$ then the claim is trivial, so assume without loss of
generality that $x< x'$.  We will first assume that $y\le y'$.

We note that no matter what the parity of $i+j$, there is always a
path of length 2 from $(i,j)$ to $(i+1,j+1)$.  Then for $1\le i \le
|y-y'|$, define $P_i$ to be the path of length 2 from
$(x+i-1,y+i-1)$ to $(x+i,y+i)$.  Then $\bigcup_{i=1}^{|y-y'|} P_i$
is a path of length $2(y'-y)$ from $(x,y)$ to $(x+y'-y, y')$.  Since $(i,j)\sim (i+1,j)$ for all $i,j$, there is a path  $P'$ of
length $x'-x+y-y'$ (Note: this number is nonnegative since
$x'-x\ge y'-y$.) from $(x+y'-y, y')$ to $(x',y')$ by simply moving
from $(i,y')$ to $(i+1,y')$ for $x+y'-y\le i< x'$.  Then, we
calculate the total length of $$P'\cup \bigcup_{i=1}^{|y-y'|}P_i$$
to be $|x'-x|+|y'-y|$.  By Claim \ref{taxicab1}, it follows that
$d(u,v)= |x-x'|+|y-y'|=\|u-v\|_1$.

If $y>y'$, then a symmetric argument follows by simply following a
downward diagonal followed by a straight path to the right.
\end{proofcite}\vsten

\begin{proofcite}{Claim~\ref{vertexlinedistance}}
We proceed by induction on $b-a$.

The base case is trivial since if $a+k$ is even, then
$(k,a)\sim(k,a+1)=(k,b)$.  If $a+k$ is odd, then there is no edge
directly to $L_b$ and so any path from $(k,a)$ to $L_b$ has length
at least 2.  This is attainable by the path from $(k,a)$ to $(k+1,a)$ to $(k+1,b)$.

The inductive step follows similarly.  Let $a+k$ be even.  Then
there is an edge between $(k,a)$ and $(k,a+1)$.  Since $k+a+1$ is
odd, the shortest path from $(k,a+1)$ to $L_b$ has length $2(b-a-1)$
and so the union of that path with the edge $\{(k,a),(k,a+1)\}$
gives a path of length $2(b-a)-1$.  Likewise, if $k+a$ is odd, then
we take the path from $(k,a)$ to $(k+a,a)$ to $(k+1,a+1)$ union a path from
$(k+1,a+1)$ to $L_b$ gives us a path of length $2(b-a)$.  Further, any path from $L_a$ to $L_b$ must contain at least one
point in $L_{a+1}$.  So if we take a path of length $\ell$ from
$(k,a)$ to $(j,a+1)$ (note that $j+a+1$ must be odd), then we get a
path of length $\ell+2(b-a-1)$.  From the base case, we have chosen
our paths from $L_a$ to $L_{a+1}$ optimally and so this path is
minimal.
\end{proofcite}\vsten

\begin{proofcite}{Claim~\ref{vertexlinedistance2}}
The proof is symmetric to the previous proof.
\end{proofcite}\vsten

\begin{proofcite}{Claim~\ref{evenvertexdistance}}
  The proof is by induction on $k$.  If $k=1$, then as noted before,
  there is always a path of length 2 from $(x,y)$ to $(x\pm 1, y\pm
  1)$.

  For $k>1$, there is a path of length 2 from $(x,y)$ to $(x-1,y+1)$
  and to $(x+1,y+1)$.  By the inductive hypothesis, there is a path of
  length $2(k-1)$ from $(x-1, y+1)$ to each vertex in
  $$S=\{(x-k+2j,y+k):j=0,1,\ldots,k-1\}$$ and a path of
  length $2(k-1)$ from $(x+1, y+1)$ to $(x+k,y+k)$.  Taking the
  union of the path of length 2 from $(x,y)$ to $(x-1,y+1)$ and the
  path of length $2(k-1)$ from $(x-1,y+1)$ to each vertex in $S$
  gives us a path of length $2k$ to each vertex in $$\{(x-k+2j,y+ k): j=0,1,\ldots,k-1\} .$$
  Then, taking the path of length 2 from $(x,y)$ to $(x+1,y+1)$ and
  the path of length $2(k-1)$ from $(x+1,y+1)$ to $(x+k,y+k)$ gives
  us a path of length $2k$ from $(x,y)$ to each vertex in $$\{(x-k+2j,y+ k): j=0,1,\ldots,k\}.$$

  By a symmetric argument, we can find paths of length $2k$ from
  $(x,y)$ to each vertex in $$\{(x-k+2j,y-k): j=0,1,\ldots,k\} .$$
\end{proofcite}\vsten

\begin{proofcite}{Claim~\ref{oddevenvertexdistance}}
Since $x+y$ is even, $(x,y)\sim(x,y+1)$.  By Claim ~\ref{evenvertexdistance}, there are paths of length $2k$ from $(x,y+1)$ to each vertex in $\{(x-k+2j,y+k+1): j=0,1,\ldots,k\}$ and hence there are paths of length $2k+1$ from $(x,y)$ to each vertex in that set.

Since $(x,y)\sim(x-1,y)$, by Claim ~\ref{evenvertexdistance}, there
are paths of length $2k$ from $(x-1,y)$ to each vertex in
$\{(x-k-1+2j,y-k): j=0,1,\ldots,k\}$.  Similarly, since
$(x,y)\sim(x+1,y)$ there is a path of length $2k$ from $(x+1,y)$ to
$(x+k+1,y-k)$ and so there is a path of length $2k+1$ from $(x,y)$
to each vertex in $\{(x-k-1+2j,y-k): j=0,1,\ldots,k+1\}$.
\end{proofcite}\vsten

\begin{proofcite}{Claim~\ref{oddoddvertexdistance}}
Since $x+y$ is odd, $(x,y)\sim(x,y-1)$.  By Claim
~\ref{evenvertexdistance}, there are paths of length $2k$ from
$(x,y-1)$ to each vertex in $\{(x-k+2j,y-k-1): j=0,1,\ldots,k\}$ and
hence there are paths of length $2k+1$ from $(x,y)$ to each vertex
in that set.

Since $(x,y)\sim(x-1,y)$, by Claim ~\ref{evenvertexdistance}, there
are paths of length $2k$ from $(x-1,y)$ to each vertex in
$\{(x-k-1+2j,y+k): j=0,1,\ldots,k\}$.  Similarly, since
$(x,y)\sim(x+1,y)$ there is a path of length $2k$ from $(x+1,y)$ to
$(x+k+1,y+k)$ and so there is a path of length $2k+1$ from $(x,y)$
to each vertex in $\{(x-k-1+2j,y+k): j=0,1,\ldots,k+1\}$.
\end{proofcite}\vsten

\begin{proofcite}{Claim~\ref{vertexballdistance}}
Without loss of generality, assume that $y\ge k$.  If $y=k$, the
Claim is trivial, so let $\ell>0$ be the length of the
shortest path between $(x,y)$ and $L_k$.  By assumption, $\ell\le
r-1$.

First assume that $\ell$ is even.  By Claim
~\ref{evenvertexdistance}, there is a path of length $\ell$ from
$(x,y)$ to each vertex in
$S=\{((x-\ell/2+2j,k):j=0,1,\ldots,\ell/2\}$.  Note that the
vertices in the set
$S'=\{((x-\ell/2+2j-1,k):j=0,1,\ldots,\ell/2+1\}$ all fall within
distance 1 of a vertex in $S$ and so $S\cup S' =
\{((x-\ell/2-1+j,k):j=0,1,\ldots,\ell+2\}\subset B_r((x,y))$.

Now note that $d((x,y),(x-\ell/2,k))=\ell$, so if $x-(r-|y-k|)\le
x_0\le x-\ell/2$ for some $x_0$, then
\begin{eqnarray*}
d((x_0,k),(x,y))&\le& d((x_0,k),(x-\ell/2,k))+d((x-\ell/2,k),(x,y))
\\
&\le& (r-|y-k|-\ell/2)+\ell .
\end{eqnarray*}
Since $|y-k|=\ell/2$ this gives $d((x_0,k),(x,y))\le r$.  Likewise,
if $x+\ell/2\le x_0\le x+(r-|y-k|)$ for some $x_0$, then
$d((x_0,k),(x,y))\le r$.  Hence, $d((x_0,k),(x,y))\le r$ for all
$x-(r-|y-k|)\le x_0\le x+(r-|y-k|)$ and so the Claim follows.

If $\ell$ is odd, the Claim follows by using either Claim
~\ref{oddevenvertexdistance} or Claim ~\ref{oddoddvertexdistance}
and applying a similar argument.

\end{proofcite}\vsten

\begin{proofcite}{Claim~\ref{codewordclaim}}

This first part of this claim is immediate from the definition of $C'$ since the only vertices in $L_{n(r+1)}$ which are not in $L_{n(r+1)}'$ are separated by distance 2.  To see the second part, let $\ell=3r$ if $r$ is even and $\ell=3r-1$ if $r$ is odd.  Write $L_{n(r+1)}'=S_1\cup S_2\cup S_3$ where
\begin{eqnarray*}
  S_1 &=& L_{n(r+1)}'\cap\{(x,y), x\equiv 1,\ldots, r-1 \mod \ell\}\\
  S_2 &=& L_{n(r+1)}'\cap\{(x,y), x\equiv r,r+1,\ldots, 2r-1 \mod \ell\}\\
  S_3 &=& L_{n(r+1)}'\cap\{(x,y), x\equiv 2r,2r+1,\ldots, \ell \mod \ell\}\\
\end{eqnarray*}
Note that each vertex in $S_2\cup S_3$ is in $C'$ and so if $u,v\in L_{n(r+1)}$ are both not in $C'$, they are both in $S_1$.  However, it is clear that for two vertices in $S_1$, their distance is either strictly less than $r$ or at least than $2r+1$.

Finally, for the last part of the claim, we use the same partition of $L_{n(r+1)}$ so that all non-codewords fall in $S_1$.  However, the distance between $(x,n(r+1))$ and $(x+2\lceil(r+1)/2\rceil)$ is at least $r+1$ so one of those two vertices cannot fall in $S_1$ and so it is a codeword.
\end{proofcite}

\section{Conclusion}

In an upcoming paper, we also provide improved lower bounds for $D(G_H,r)$ when $r=2$ or $r=3$~\cite{MartinSubmitted}.  Below are a couple tables noting our improvements.  This includes
the results not only from our paper, but also from the aforementioned paper.

$$\begin{array}{c}

\begin{array}{|c|c||c|c|}
  \hline
  \multicolumn{4}{|c|}{\text{Hex Grid}}\\
  \hline
  r & \text{previous lower bounds} & \text{new lower bounds} & \text{upper bounds} \\
  \hline
  2 & 2/11\approx 0.1818 ^\text{~\cite{Karpovsky1998}} & 1/5 = 0.2^\text{~\cite{MartinSubmitted}} & 4/19\approx 0.2105 ^\text{~\cite{Charon2002}}\\
  \hline
  3 & 2/17\approx 0.1176 ^\text{~\cite{Charon2001}} & 3/25 = 0.12^\text{~\cite{StantonPending}} & 1/6\approx 0.1667 ^\text{~\cite{Charon2002}}\\
  \hline
  \multicolumn{4}{|c|}{\text{Square Grid}}\\
  \hline
  2 & 3/20=0.15 ^\text{~\cite{Charon2001}} & 6/37\approx 0.1622^\text{~\cite{MartinSubmitted}} & 5/29\approx 0.1724 ^\text{~\cite{Honkala2002}} \\
  \hline
\end{array}\\
\vspace{1em}\\

\begin{array}{|c|c|c||c|}
  \hline
  \multicolumn{4}{|c|}{\text{Hex Grid}}\\
  \hline
  r & \text{lower bounds} & \text{new upper bounds} & \text{previous upper bounds} \\
  \hline
  15 & 2/77\approx 0.0260 ^\text{~\cite{Charon2001}} & 1227/22528\approx  0.0523 & 1/18\approx 0.0556 ^\text{~\cite{Charon2002}} \\
  \hline
  16 & 2/83\approx 0.0241 ^\text{~\cite{Charon2001}} & 83/1632\approx  0.0509    & 1/18\approx 0.0556^\text{~\cite{Charon2002}}\\
  \hline
  17 & 2/87\approx 0.0230 ^\text{~\cite{Charon2001}} &                           & 1/22\approx 0.0455^\text{~\cite{Charon2002}}\\
  \hline
  18 & 2/93\approx 0.0215 ^\text{~\cite{Charon2001}} & 31/684 \approx  0.0453    & 1/22\approx 0.0455 ^\text{~\cite{Charon2002}}\\
  \hline
  19 & 2/97\approx 0.0206 ^\text{~\cite{Charon2001}} & 387/8960\approx  0.0432    & 1/22\approx 0.0455 ^\text{~\cite{Charon2002}}\\
  \hline
  20 & 2/103\approx 0.0194 ^\text{~\cite{Charon2001}} & 103/2520\approx  0.0409    & 1/22\approx 0.0455 ^\text{~\cite{Charon2002}}\\
  \hline
  21 & 2/107\approx 0.0187 ^\text{~\cite{Charon2001}} &                            & 1/26\approx 0.0385 ^\text{~\cite{Charon2002}}\\
  \hline
\end{array}
\end{array}$$

For even $r\ge 22$, we have improved the bound from approximately $8/(9r)$~\cite{Charon2001} to $(5r+3)/(6r(r+1))$ and for odd $r\ge 23$ we have improved the bound from approximately $8/(9r)$~\cite{Charon2001} to $(5r^2+7r-3)/((6r-2)(r+1)^2)$.

Constructions for $2\le r\le 30$ given in
~\cite{Charon2002} were previously best known.  However, the best general
constructions were given in ~\cite{Charon2001}.

\bibliographystyle{plain}
\bibliography{../bibtex/bibtex}

\end{document}